\documentclass[pdflatex,sn-mathphys-num]{sn-jnl}

\usepackage{graphicx}%
\usepackage{multirow}%
\usepackage{amsmath,amssymb,amsfonts}%
\usepackage{amsthm}%
\usepackage{mathrsfs}%
\usepackage[title]{appendix}%
\usepackage{xcolor}%
\usepackage{textcomp}%
\usepackage{manyfoot}%
\usepackage{booktabs}%
\usepackage{algorithm}%
\usepackage{algorithmicx}%
\usepackage{algpseudocode}%
\usepackage{listings}%
\usepackage{enumitem}
\usepackage{hyperref}
\usepackage{xcolor}
\hypersetup{
    colorlinks,
    linkcolor={red!50!black},
    citecolor={blue!50!black},
    urlcolor={blue!80!black}
}
\usepackage[nameinlink,capitalize,noabbrev]{cleveref}
\usepackage[rightcaption]{sidecap}
\sidecaptionvpos{figure}{c}
\usepackage{tikz-cd}

    \newcommand{\R}
    {
        \mathbb{R}
    }
    
    \renewcommand{\dot}[1]
    {
    \overset{\raisebox{0ex}{\scalebox{0.45}{$\bullet$}}}{#1}
    }
    
    \newcommand{\restr}[2]
    {
    \left.\kern-\nulldelimiterspace
    #1 
    \vphantom{\big|}
    \right|_{#2}
    }

\theoremstyle{thmstyleone}%
\newtheorem{theorem}{Theorem}
\newtheorem{lemma}[theorem]{Lemma}

\newtheorem{question}[theorem]{Question}
\newtheorem{corollary}[theorem]{Corollary}

\theoremstyle{thmstyletwo}%
\newtheorem{example}{Example}%
\newtheorem{remark}{Remark}%

\theoremstyle{thmstylethree}%
\newtheorem{definition}{Definition}%

\begin{document}

\title[Article Title]{Brockett Openness Profiles and Gain-Limited Feedback Stabilization}

\author*[1]{\fnm{Bryce} \sur{Christopherson}}\email{bryce.christopherson@UND.edu}
\equalcont{These authors contributed equally to this work.}

\author[2]{\fnm{Farhad} \sur{Jafari}}\email{fjafari@UMN.edu}
\equalcont{These authors contributed equally to this work.}

\affil*[1]{\orgdiv{Department of Mathematics \& Statistics}, \orgname{University of North Dakota}, \orgaddress{\street{101 Cornell Street}, \city{Grand Forks}, \postcode{58202}, \state{North Dakota}, \country{United States}}}

\affil[2]{\orgdiv{Department of Radiology, Medical Physics Division}, \orgname{University of Minnesota Medical School}, \orgaddress{\city{ Minneapolis}, \postcode{55455}, \state{Minnesota}, \country{United States}}}


\abstract{Brockett's necessary condition asserts that a continuously stabilizable nonlinear control system must have a vector field that is open at the equilibrium.  We show that the quantitative data behind this openness condition constrains the possible growth of stabilizing feedbacks.  To a system vector field $f$, we associate its openness profile $\Omega_f(r)=\sup\{\rho:\mathbb{B}_\rho(0)\subset f(\mathbb{B}_r(0,0))\}$, so that Brockett's condition becomes $\Omega_f(r)>0$ for all sufficiently small $r>0$.  If a feedback $u$ satisfies $\|u(x)\|\leq d(\|x\|)$, then the openness profile of the closed-loop field $F_u(x)=f(x,u(x))$ satisfies $\Omega_{F_u}(r)\leq \Omega_f\!\left(\sqrt{r^2+d(r)^2}\right)$.  Consequently, any prescribed lower openness rate for the closed-loop dynamics yields a necessary lower bound on the feedback growth.  For systems with $\Omega_f(r)\lesssim r^q$, linear-rate closed-loop openness forces $d(r)\gtrsim r^{1/q}$, and this exponent is sharp in elementary polynomial examples.  Thus Brockett's condition is not merely a binary topological obstruction; its quantitative profile governs gain requirements for stabilizing feedback.}




\maketitle


\section{Introduction}


    Brockett's theorem is one of the central topological obstructions to continuous feedback stabilization. It says that if a nonlinear control system is continuously stabilizable at an equilibrium, then the system vector field must be open at that equilibrium. This condition is usually used as a binary test. In this paper we extract quantitative information from the same datum.  Specifically, we will use this quantitative refinement to address the following general question:
    
    
    \begin{question}[Gain-Limited Stabilization Problem]\label{question}
        Consider a control system of the form 
        
        
        \begin{align}
            \dot{x}=f(x,u), & \enskip t \geq 0, \label{control sys}
        \end{align}
        
        
        \noindent which is locally asymptotically stabilizable at the equilibrium point $x_e$.  Let $d:[0,\epsilon)\to[0,\infty)$ be a nondecreasing function with $d(r)\to 0$ as $r \to 0$. Under what conditions does there exist a feedback controller $u(x)$ such that $u(x)$ locally asymptotically stabilizes \eqref{control sys} at $x_e$ and satisfies 
        \begin{equation}\label{question rate}
            \|u(x)\| \leq d(\|x - x_e\|)
        \end{equation} for all $x$ sufficiently close to $ x_e$?
    \end{question}
    
    
    Notice, Question~\ref{question} is simple in the case of $d$ a constant function and $u$ continuous.  That is, it is always possible:  If we assume $u(x)$ is continuous and stationary and that \eqref{control sys} can be stabilized at $x_e$, then we can stabilize the system with a control of arbitrarily small norm on a sufficiently small ball around $x_e$.  In this case, what is usually studied is the trade off between the size of this ball and the norm bound on the control.  While interesting in its own right, this formulation of the problem fails to capture the fine-grained details regarding the growth rate of the control near the equilibrium, which is often of practical interest.  Additionally, the tradeoff in this problem is `non-local' in the sense that it depends on characterizations of the system's trajectories behavior away from the equilibrium point. 
    
    The purpose of this paper is not to solve the gain-limited stabilization problem in full generality.  Rather, we derive necessary conditions on \(d\) that follow from Brockett's openness obstruction.  Brockett's theorem is usually read as a binary obstruction: the vector field is either open at the equilibrium or it is not.  We show that, once openness holds, the quantitative rate at which the image of small neighborhoods fills a neighborhood of the origin constrains the possible growth and decay rates of any stabilizing feedback.  Thus Brockett-type data does not merely decide whether stabilization is possible; it also constrains the gain profile of stabilizing feedbacks.


\section{Background and Technical Preliminaries}

    
    We consider autonomous control systems of the form \eqref{control sys}.  More specifically, let  $\mathcal{X} \times \mathcal{U} \subseteq \R^n \times \R^m$ be a neighborhood of the origin. Unless otherwise stated, we assume that the function $f$ on the right-hand side of \eqref{control sys} satisfies both conditions $f(0,0)=0$ and $f \in \mathrm{C}\left (\mathcal{X} \times \mathcal{U}, \R^n\right )$.

    Here, we are interested in {\it stabilizing} these systems. That is (as in, e.g., \cite[Definition 10.11]{coron1}), we say the system is \textit{locally asymptotically stabilizable by means of feedback laws} if there exists a neighborhood of the origin $\mathcal{O}\subseteq \mathcal{X}$ and a feedback $u  \colon  \mathcal{O} \rightarrow \mathcal{U}$ satisfying $u(0)=0$ that renders the origin a locally asymptotically stable equilibrium \cite[Definition 2.1]{byrnes} of the closed-loop system $\dot{x}=f\big (x,u(x) \big )$.  For clarity, we restrict ourselves here to feedbacks yielding a continuous closed-loop vector field with unique forward solutions; this aligns with, e.g.,  Coron and Zabczyk \cite{coron2,zab}, but avoids Filippov or Krasovskii notions of solutions. Our results therefore apply to continuous stabilizers in the classical sense. 

    
    \begin{remark}
        Note that the requirement that $u(0)=0$ is a mostly-harmless simplification and is no more of an imposition than the requirement that $f(0,0)=0$.  That is, if we are in some situation where we stabilize our system $\dot{x}=f(x,u)$ with a feedback $u(x)$ such that $u(0)=u_0\neq 0$, then we must have $f(0,u_0) = 0$.  So, we could just consider the system $\tilde{f}(x,u) := f(x,u+u_0)$ instead, which certainly has $\tilde{f}(0,0) = 0$ and is stabilized by the feedback law $\tilde{u}(x) := u(x)-u_0$ that satisfies $\tilde{u}(0)=0$.
    \end{remark}
    
    
    A variety of conditions describing whether the system \eqref{control sys} is locally asymptotically stabilizable by means of continuous feedback laws have been derived. As a nowhere-near-comprehensive listing, see, e.g., the control Lyapunov techniques described in \cite{artstein,Lafferriere,braun,goebel}, the topological conditions in \cite{brockett,byrnes,coron2,coron1, gupta2017linear, christopherson2019, christopherson2022continuous}, the homological conditions in \cite{coron2}, standard Lyapunov-theoretic techniques derived from solutions to a Hamilton-Jacobi-Bellman equation in \cite{hermes}, polynomial stabilization in \cite{jammazi}, and compilations of many others in \cite{coron1, sontag1,sontag2,onishchenko1, onishchenko2}.  We use none of these more sophisticated approaches and work directly with Brockett's well-known openness condition and the graph map $x\mapsto (x,u(x))$.
    

\section{Openness Profiles and Brockett's Condition}

    
    A necessary condition of Brockett \cite{brockett} for feedback stabilization is that the vector field inducing the system must satisfy a certain `local openness' property at the equilibrium for such a controller to exist.
    
    \begin{definition}[Locally Open]\label{open}
        Let $X$ and $Y$ be topological spaces.  A mapping $f \colon X\rightarrow Y$ is said to be \emph{open at  $x\in X$} if we have $f(x) \in \textrm{int}f(\mathcal{O})$ for any neighborhood $\mathcal{O}$ of $x$.
    \end{definition}
    
    The origin of this local openness property goes back to the classical Banach-Schauder open mapping theorem of functional analysis, though it should be noted that the class of \textit{open mappings} is a more restrictive one than that used by Brockett.  As refined by Coron \cite{coron2} and Zabczyk \cite{zab}, Brockett's obstruction is the following one: 
    
    
     \begin{theorem}[Brockett's Theorem - Continuous Extension] \label{brocketts - continuous extension} 
        If the system \eqref{control sys} is locally asymptotically stabilizable by means of continuous feedback laws, then it is necessary that $f$ is open at the origin.
    \end{theorem}
    
    
    While a mapping may have the openness property at a given point, this property alone does not admit any measuring device to quantify the rate at which the mapping  is open about that point. Some effort has been devoted to using the variational-analytic property of \textit{linear openness} (also referred to as {\it covering} or {\it openness at a linear rate}, see, e.g., \cite{mordukhovich3,rockafellar}) in conjunction with Brockett's criterion \cite{christopherson2022continuous,christopherson2019}.  The idea behind linear openness is to provide a quantitative measure of `how open' a mapping is at a given point.  Here, we introduce a slightly different device for quantifying the rate of local openness.  Variants of this notion have appeared in several other contexts  and under different names; see, for example,  \cite{ioffe2013nonlinear,borwein1988verifiable}.  We continue that tradition here. 

    
    \begin{definition}\label{cont open def}
        Let $X$, $Y$ be metric spaces and let $f:X \rightarrow Y$.  For $x \in X$ and a continuous, positive definite, monotonically increasing function $ g $, we say that \emph{$f$ is continuously open at $x$ with rate $g$} if $\mathbb{B}_{g(r)}\big(f(x)\big)\subseteq f(\mathbb{B}_r(x))$ for all sufficiently small $r>0$.  Likewise, we say $f$ is \emph{continuously open at $x$} if $f$ is continuously open for some rate function satisfying the above.
    \end{definition}


    This definition allows us to measure how rapidly the image of a shrinking ball must expand, which we will show directly determines `how much' control is required to stabilize the system.
    
    
    \begin{example}\label{bconstant example}
        Suppose $T$ is a surjective bounded linear map between Banach spaces. Then, for every $0<\kappa<\Gamma(T)$, $T$ is open at the origin at the linear rate $g(r)=\kappa r$, where $\Gamma(T)$ is the Banach constant of $T$; i.e., the quantity $\Gamma(T):=\mathrm{sup}\big\{r\geq 0\;\big|\;\mathbb{B}_r\subseteq T\big(\mathbb{B}\big)\big \}$.

        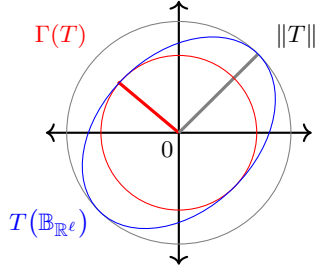
\begin{SCfigure}[2.5][ht]
            \caption{The Banach constant $\Gamma(T)$ and norm $\|T\|$ of a linear operator $T$ shown as the radii of the balls contained in and containing the $T$-image of the unit ball $T(\mathbb{B})$, respectively.} \begin{tikzpicture}[scale=0.5]
                \draw (1.7,2) node[below]{$0$};
                \draw[text=red] (-1.1,5.1) node[below]{$\Gamma(T)$};
                \draw (5.1,5.1) node[below]{$\|T\|$};
                \draw[text=blue] (-1.4,0.2) node[below]{$T\big (\mathbb{B}_{\R^\ell} \big )$};
                \draw[<->, thick] (2,-1.5) -- (2,5.5);
                \draw[<->, thick] (-1.5,2) -- (5.5,2);
                \draw[very thick,red] (2,2) -- (0.4,3.35); 
                \draw[very thick,gray] (2,2) -- (4.09,4.09);
                \draw[gray] (2,2) circle (2.95cm);
                \draw[red] (2,2) circle (2.05cm);
                \draw[blue,rotate around={-45:(2,2)}] (2,2) ellipse (2.05cm and 2.95cm);
            \end{tikzpicture}\label{b constant fig} 
        \end{SCfigure}
    \end{example}
    
    Continuous openness is closely related to the classical notion of linear openness in variational analysis, which first appeared in \cite{milyutin} under the name of ``covering in a neighborhood" and has since been substantially developed \cite{mordukhovich2,rockafellar,ioffe}.  In finite dimensions, \(f\) is linearly open at the origin with modulus \(\kappa>0\) precisely when
    \[
    \mathbb{B}_{\kappa r}(0)\subseteq f(\mathbb{B}_r(0))
    \]
    for all sufficiently small \(r>0\).  It is hopefully straightforward to see that linear openness yields its openness counterpart from Definition~\ref{open}, while the opposite implication fails. The point of Definition~\ref{cont open def} is to retain nonlinear rates, which are invisible to a purely linear openness modulus.
    
    Unlike linear openness, however, continuous openness at a point is, in fact, equivalent to openness at a point for continuous mappings (cf. Theorem~\ref{open cont open equiv}), so this tool is always available to us when studying stabilizable systems.  As suggested by Example~\ref{bconstant example}, we adopt the following notation:
    

    \begin{definition}\label{b constant def}
        Let $X$ be a normed space, let $x^* \in X$, and let $A\subseteq X$.  We say the \emph{inscribed radius} (or \emph{inradius}) \emph{of $A$ at $x^*$} is the quantity
        $$
            \Gamma_{x^*}(A)=\sup \{r\geq 0:\mathbb{B}_r(x^*)\subseteq A\}.
        $$
    \end{definition}
    
    
    Now, we will use this notion to show that continuous openness at a point and openness at a point are, in fact, equivalent.

    \begin{theorem}\label{open cont open equiv}
        Let $X$, $Y$ be normed vector spaces and suppose $f:X \rightarrow Y$.  For $x^*\in X$, define
        $$
            \Omega_{f,x^*}(r):=\Gamma_{f(x^*)}\big(f(\mathbb{B}_r(x^*))\big).
        $$
        Then $f$ is open at $x^*$ if and only if $\Omega_{f,x^*}(r)>0$ for every sufficiently small $r>0$.  Moreover, if $f$ is open at $x^*$, then $f$ is continuously open at $x^*$.  Finally, $\Omega_{f,x^*}$ is maximal in the following sense: if $f$ is continuously open at $x^*$ at a rate of $h$, then $h(r)\leq \Omega_{f,x^*}(r)$ for all sufficiently small $r>0$.
    \end{theorem}


    \begin{proof}
        Suppose first that \(\Omega_{f,x^*}(r)>0\) for every sufficiently small \(r>0\).  Let $\mathcal{O}$ be a neighborhood of \(x^*\). Choose \(r>0\) sufficiently small so that $\mathbb{B}_r(x^*)\subset \mathcal{O}$ and $\Omega_{f,x^*}(r)>0$. Then there exists $\rho>0$ such that
        \[
            B_\rho(f(x^*))\subset f(B_r(x^*))\subset f(\mathcal{O}).
        \]
        Hence \(f(x^*)\in \operatorname{int} f(\mathcal{O})\), and so \(f\) is open at \(x^*\).

        Conversely, suppose that $f$ is open at $x^*$.  Then $f(\mathbb{B}_r(x^*))$ is a neighborhood of $f(x^*)$ for every sufficiently small $r>0$.  Thus there exists some $\rho>0$ such that $\mathbb{B}_{\rho}\big(f(x^*)\big)\subseteq f(\mathbb{B}_r(x^*))$.  Hence $\Omega_{f,x^*}(r)>0$ for every sufficiently small $r>0$.

        Now, we show the function $\Omega_{f,x^*}$ is monotonically increasing.  Indeed, if $0<r_1<r_2$, then $\mathbb{B}_{r_1}(x^*)\subseteq \mathbb{B}_{r_2}(x^*)$, so $f(\mathbb{B}_{r_1}(x^*))\subseteq f(\mathbb{B}_{r_2}(x^*))$.  Therefore every ball centered at $f(x^*)$ contained in $f(\mathbb{B}_{r_1}(x^*))$ is also contained in $f(\mathbb{B}_{r_2}(x^*))$, and so $\Omega_{f,x^*}(r_1)\leq \Omega_{f,x^*}(r_2)$.

        It remains to construct a continuous rate.  Fix $R>0$ sufficiently small so that $\Omega_{f,x^*}(r)>0$ for every $0<r\leq R$.  Set $r_j:=2^{-j}R$ for $j=0,1,2,\ldots$.  Choose positive numbers $a_j$ recursively so that $0<a_0<\Omega_{f,x^*}(r_1)$ and, for $j\geq 1$,
        $$
            0<a_j<\min\left\{a_{j-1},2^{-j},\Omega_{f,x^*}(r_{j+1})\right\}.
        $$
        Define $g(0)=0$, set $g(r_j)=a_j$, and extend $g$ linearly on each interval $[r_{j+1},r_j]$.  Then $g$ is continuous, positive definite, and monotonically increasing on $[0,R]$.  Moreover, if $r\in [r_{j+1},r_j]$, then $g(r)\leq a_j<\Omega_{f,x^*}(r_{j+1})\leq \Omega_{f,x^*}(r)$.  Therefore
        $$
            \mathbb{B}_{g(r)}\big(f(x^*)\big)\subseteq f(\mathbb{B}_r(x^*))
        $$
        for every $0<r\leq R$.  Hence $f$ is continuously open at $x^*$.

        Finally, suppose $f$ is continuously open at $x^*$ at a rate of $h$.  Then, $\mathbb{B}_{h(r)}\big(f(x^*)\big)\subseteq f(\mathbb{B}_r(x^*))$ for all sufficiently small $r>0$.  By the definition of $\Omega_{f,x^*}(r)$ as the supremum of all radii $\rho$ for which $\mathbb{B}_{\rho}\big(f(x^*)\big)\subseteq f(\mathbb{B}_r(x^*))$, we immediately have $h(r)\leq \Omega_{f,x^*}(r)$ for all sufficiently small $r>0$.
    \end{proof}
    

    Accordingly, we define the \textit{openness profile} of a vector field using this maximal openness profile.  Continuous openness rates should be understood as continuous positive minorants of this profile.
    

    \begin{definition}\label{openness profile definition}
        To a system vector field $f$ in \eqref{control sys}, we associate its \textit{openness profile} 
        \[
        \Omega_f(r)=\Gamma_{f(0,0)}\big(f\big(\mathbb{B}_r(0,0)\big)\big)=\sup\{\rho:\mathbb{B}_\rho(0)\subset f(\mathbb{B}_r(0,0))\}.
        \]
    \end{definition}

    
    In this notation, Brockett's condition becomes the following:

    
    \begin{lemma}
        Let $f:\mathbb{R}^{n+m}\to\mathbb{R}^n$ be continuous with $f(0,0)=0$. Then $f$ is open at the origin if and only if $\Omega_f(r)>0$ for every sufficiently small $r>0$.
    \end{lemma}
    
    
    \begin{proof}
        If $\Omega_f(r)>0$ for every sufficiently small $r>0$, then for each such $r$ there exists $\rho>0$ such that $\mathbb{B}_\rho(0)\subset f(\mathbb{B}_r(0,0))$.  Thus $f(\mathbb{B}_r(0,0))$ is a neighborhood of $0=f(0,0)$ for every sufficiently small $r$, and $f$ is open at the origin.  Conversely, if $f$ is open at the origin, then $f(\mathbb{B}_r(0,0))$ is a neighborhood of the origin in $\mathbb{R}^n$ for every sufficiently small $r>0$.  Hence there exists some $\rho>0$ such that $\mathbb{B}_\rho(0)\subset f(\mathbb{B}_r(0,0))$, which says precisely that $\Omega_f(r)>0$.
    \end{proof}
    

    \begin{example}\label{nonholonomic integrator example}
        Consider Brockett's nonholonomic integrator $f(x,u)=\begin{bmatrix} u_1 & u_2 & x_1u_2-x_2u_1 \end{bmatrix}^T$. This is a classic example of a nonholonomic system that fails to have a continuous stabilizing feedback. To wit, observe that if the first two components of $f(x,u)$ vanish, then $u_1=u_2=0$, and therefore the third component $x_1u_2-x_2u_1$ also vanishes.  Hence $f(\mathbb{B}_r(0,0))$ cannot contain any point of the form $(0,0,\epsilon)$ with $\epsilon\neq0$.  Consequently, no ball centered at $0$ in $\mathbb{R}^3$ is contained in $f(\mathbb{B}_r(0,0))$, and $\Omega_f(r)=0$.  Thus the openness profile recovers the usual Brockett obstruction in this case.
    \end{example}

    
    The benefit of this definition is that it allows us to say more than this as well:  quantitatively, the openness profile tells us how fast $\Omega_f(r)$ vanishes as $r\rightarrow 0$.


\section{Control Limitations from Brockett Datum}


Using the graph map $x \mapsto \big(x,u(x)\big)$, we get an easy necessary condition for Question~\ref{question} to be answered in the affirmative.  The strength of this elementary inequality is that it converts estimates on the open-loop Brockett profile into explicit lower bounds on admissible feedback gain profiles.

    
    \begin{theorem}[Graph-limited Brockett inequality]\label{graph limited brocket inequality}
        Let $f:\mathbb{R}^n\times\mathbb{R}^m\to\mathbb{R}^n$ be continuous with $f(0,0)=0$. Let $u$ be a feedback defined near $0$ with $u(0)=0$, and set $F_u(x):=f(x,u(x))$.  Let $d:[0,\epsilon)\to[0,\infty)$ be nondecreasing, and assume $\|u(x)\|\leq d(\|x\|)$ for $x$ near $0$.  Then, for all sufficiently small $r>0$, 
        \[
            \Omega_{F_u}(r) \leq \Omega_f\left(\sqrt{r^2+d(r)^2}\right).
        \]
    \end{theorem}

    
    \begin{proof}
        Write $G_u(x):=(x,u(x))$, so $F_u=f \circ G_u$.  If $x \in \mathbb{B}_r(0)$ and $r>0$ is sufficiently small, then the monotonicity of $d$ gives
        \[
            \|G_u(x)\|=\sqrt{\|x\|^2+\|u(x)\|^2}\leq \sqrt{r^2+d(r)^2}.
        \]
        Hence $G_u\big(\mathbb{B}_r(0)\big) \subseteq \mathbb{B}_{\rho(r)}(0,0)$ for $\rho(r):=\sqrt{r^2+d(r)^2}$.  Therefore,
        \[
            F_u\big(\mathbb{B}_r(0)\big)=f\big(G_u\big(\mathbb{B}_r(0)\big)\big) \subseteq f\big(\mathbb{B}_{\rho(r)}(0,0)\big).
        \]
        Recasting in the language of Definition~\ref{openness profile definition}, the result follows.
    \end{proof}
    

    To summarize, Theorem~\ref{graph limited brocket inequality} tells us that if the feedback is gain-limited by $d$, then the closed-loop vector field cannot have more local openness than the open-loop vector field can supply along the feedback graph. As a consequence, if one has a desired closed-loop system, then its openness profile must be achievable with the constraint posed by the gain limitation $d$ in Question~\ref{question}.  We record this in the following corollary.
    
    
    \begin{corollary}[Minimum gain from closed-loop openness]\label{minimum gain from closed-loop openness}
        If, in addition, $\Omega_{F_u}(r)\geq \gamma(r)$ for all sufficiently small $r>0$, then every gain bound $d$ satisfied by $u$ must obey
        \begin{equation}\label{minimum gain inequality}
            \gamma(r) \leq \Omega_f\left(\sqrt{r^2+d(r)^2}\right)
        \end{equation}
        for all sufficiently small $r>0$.
    \end{corollary}


    \begin{proof}
        This follows immediately from Theorem~\ref{graph limited brocket inequality} and the assumed lower bound $\gamma(r)\leq \Omega_{F_u}(r)$.
    \end{proof}


    In short, Brockett’s theorem requires the closed-loop vector field to be open. If the closed-loop vector field has a minimum prescribed openness rate, then the system vector field must be able to realize that same openness rate along the graph of a feedback satisfying the gain bound. This forces a lower bound on the size of the satisfiable growth rate bounds $d$.

    
    \begin{remark}
        It is tempting to interpret the lower bound $\gamma$ in Corollary~\ref{minimum gain from closed-loop openness} as arising from a lower constraint on the gain of a stabilizing feedback, analogous to \eqref{question rate}; that is, from a condition of the form $\ell(x)\leq \|u(x)\|$. Such a constraint could, for instance, model limited actuator precision. Unfortunately, a lower bound of this type does not by itself imply any corresponding lower bound on the closed-loop openness profile. For example, for $f(x,u)=x+u$ and $u_p(x)=-x-x^p$, $p\geq 3$ odd, we have
        \[
        \frac{3}{4}|x|\leq |u_p(x)|\leq \frac{5}{4}|x|
        \]
        near $0$, uniformly in $p$, while the closed-loop field is $F_{u_p}(x)=-x^p$, whose openness profile is $\Omega_{F_{u_p}}(r)=r^p$. That is, there are stabilizable systems with arbitrarily weak closed-loop openness despite uniform lower bounds on the feedback magnitude.  For this reason the present paper treats lower closed-loop openness as an explicit performance hypothesis, as discussed in the next section.
    \end{remark}


\section{Openness Profile Calculus}


    In general, the lower profile $\gamma$ in Corollary~\ref{minimum gain from closed-loop openness} should not be interpreted as additional open-loop data.  Rather, it represents the desired strength of the closed-loop dynamics. The main theorem (i.e. Theorem~\ref{graph limited brocket inequality}) bounds the closed-loop Brockett profile $\Omega_{F_u}$ attainable by any feedback satisfying the gain constraint. Consequently, whenever a design objective imposes a lower bound $\Omega_{F_u}(r)\geq \gamma(r)$, the gain constraint must be large enough to satisfy the resulting inequality.
    
    In this section, we will elucidate some of the calculus of openness profiles, recording their corresponding consequences on the admissible gains of stabilizing controls.  For example, if the desired closed-loop vector field is $C^1$ with a Hurwitz linearization at the equilibrium, then $\Omega_{F_u}(r)\geq cr$ for some $c>0$. Thus the case $\gamma(r)=cr$ corresponds to the usual smooth exponential-stability regime.


    \begin{lemma}\label{linear openness profile}
        If $F:\mathbb{R}^{k}\to\mathbb{R}^{n}$ is $C^1$, $F(0)=0$, and $DF(0)$ has full row rank, then there exist constants $c,C>0$ such that
        $$
            cr\leq \Omega_F(r)\leq Cr
        $$
        for all sufficiently small $r>0$.
    \end{lemma}


    \begin{proof}
        The upper bound follows from differentiability.  Indeed, after shrinking to a sufficiently small neighborhood of $0$, there is a constant $C>0$ such that $\|F(z)\|\leq C\|z\|$ (since $C^1$ mappings are locally Lipschitz).  Hence $F(\mathbb{B}_r(0))\subseteq \mathbb{B}_{Cr}(0)$, and therefore $\Omega_F(r)\leq Cr$.

        For the lower bound, since $DF(0)$ has full row rank, $F$ is a submersion at $0$.  By the constant rank theorem, after restricting to sufficiently small neighborhoods and making $C^1$ changes of coordinates in the domain and codomain, $F$ is locally equivalent to the projection $(y,w)\mapsto y$.  These coordinate changes are locally bi-Lipschitz.  Therefore there is a constant $c>0$ such that $\mathbb{B}_{cr}(0)\subseteq F(\mathbb{B}_r(0))$ for all sufficiently small $r>0$.  Thus $\Omega_F(r)\geq cr$.
    \end{proof}
    


    \begin{corollary}\label{hurwitz gives linear profile}
        If $F:\mathbb{R}^{n}\to\mathbb{R}^{n}$ is $C^1$, $F(0)=0$, and $DF(0)$ is invertible, then $\Omega_F(r)\geq cr$ for some $c>0$ and all sufficiently small $r>0$.  In particular, if $DF(0)$ is Hurwitz, then the usual $C^1$ exponentially stable closed-loop case has a linear lower openness profile.
    \end{corollary}


    \begin{proof}
        This is the square full-rank case of Lemma~\ref{linear openness profile}.  If $DF(0)$ is Hurwitz, then it is in particular invertible.
    \end{proof}


    A higher-order variant of Lemma~\ref{linear openness profile} follows similarly.
    
    \begin{lemma}\label{flat upper profile}
        If $\|f(x,u)\|\leq C\|(x,u)\|^q$ near the origin, then $\Omega_f(r)\leq Cr^q$ for all sufficiently small $r>0$.
    \end{lemma}


    \begin{proof}
        For sufficiently small $r>0$, the assumption implies $f(\mathbb{B}_r(0,0))\subseteq \mathbb{B}_{Cr^q}(0)$.  Hence no ball centered at $0$ of radius larger than $Cr^q$ can be contained in $f(\mathbb{B}_r(0,0))$, so $\Omega_f(r)\leq Cr^q$.
    \end{proof}

    \begin{example}\label{signed polynomial example}
        Suppose $f(x,u)=\operatorname{sgn}(x)|x|^q+\operatorname{sgn}(u)|u|^q$ with $q>1$ and suppose we seek a closed-loop vector field whose openness profile satisfies $\Omega_{F_u}(r) \geq r$; this is the natural profile of a linear-rate closed loop such as $F_u(x)=-x$.  Since $|f(x,u)|\leq |x|^q+|u|^q\leq 2\|(x,u)\|^q$, Lemma~\ref{flat upper profile} gives $\Omega_f(r)\leq 2r^q$.  Corollary~\ref{minimum gain from closed-loop openness} then yields
        \[
            r \leq 2\left(r^2+d(r)^2\right)^{q/2}.
        \]
        Therefore
        \[
            \sqrt{\left(\frac{r}{2}\right)^{2/q}-r^2}\leq d(r).
        \]
        For small $r$ and $q>1$, the term $\left(\frac{r}{2}\right)^{2/q}$ dominates $r^2$, so this forces $d(r) \gtrsim r^{1/q}$.  That is, a system with an openness profile that grows on the order of $r^q$ can only be stabilized to a closed-loop system with a linear openness profile by controls $u$ satisfying $\|u(x)\| \gtrsim \|x\|^{1/q}$.
    \end{example}
    
    
    The above example suggests the following corollary, which we now record.


    \begin{corollary}\label{power law}
        Suppose that $\Omega_f(r)\leq Cr^q$ for all sufficiently small $ r >0 $.  If a feedback $u$ satisfying \eqref{question rate} produces a closed-loop field satisfying $\Omega_{F_u}(r)\geq cr^p$ near $0$, then every gain bound $d$ satisfied by $u$ must obey
        \begin{equation}\label{power law inequality}
            d(r)\geq \sqrt{\left(\frac{c}{C}\right)^{2/q}r^{2p/q}-r^2}
        \end{equation}
        for all sufficiently small $r>0$ for which the expression is real. In particular, $d(r)\gtrsim r^{p/q}$ whenever $p<q$. Consequently, if $d(r)=O(r^\beta)$, then necessarily $\beta\leq \frac{p}{q}$.  Moreover, this bound is sharp.
    \end{corollary}

    
    \begin{proof}
        By Corollary~\ref{minimum gain from closed-loop openness} with $\gamma(r)=cr^p$,
        \[
        cr^p \leq C\left(r^2+d(r)^2\right)^{q/2}.
        \]
        So, $\sqrt{r^2+d(r)^2} \geq \left(\frac{c}{C}\right)^{1/q}r^{p/q}$ and the \eqref{power law inequality} follows from basic algebra.  If $p<q$, then the first term inside the square root in \eqref{power law inequality} dominates $r^2$ as $r \to 0$, so $d(r)\gtrsim r^{p/q}$.  Therefore $d(r)=O(r^\beta)$ is possible only if $\beta\leq \frac{p}{q}$.

        For sharpness, let the vector field $f$ in \eqref{control sys} be $f(x,u)=\operatorname{sgn}(u)|u|^q$.  Then $\Omega_f(r)=r^q$ and, for any $p>0$, the feedback $u(x)=-\operatorname{sgn}(x)|x|^{p/q}$ produces $F_u(x)=-\operatorname{sgn}(x)|x|^p$, so that $\Omega_{F_u}(r)=r^p$.   Thus the lower bound $d(r)\gtrsim r^{p/q}$ is attained, up to constants, by the feedback bound $|u(x)|\leq |x|^{p/q}$.
    \end{proof}
    

    A natural source of power-law openness profiles is homogeneity.
    

    \begin{lemma}\label{homogeneous exact profile}
        Let $H:\mathbb{R}^{n+m}\to\mathbb{R}^n$ be a continuous mapping that is homogeneous of degree $q>0$, and open at the origin. Then
        $$
            \Omega_H(r)=r^q\Omega_H(1)
        $$
        for every $r>0$ for which the profile is defined.
    \end{lemma}


    \begin{proof}
        Since $H$ is homogeneous mapping of degree $q$, we have $H(\mathbb{B}_r(0))=r^qH(\mathbb{B}_1(0))$.  Therefore $\mathbb{B}_\rho(0)\subseteq H(\mathbb{B}_r(0))$ if and only if $\mathbb{B}_{\rho/r^q}(0)\subseteq H(\mathbb{B}_1(0))$.  Taking suprema over such $\rho$ gives $\Omega_H(r)=r^q\Omega_H(1)$.
    \end{proof}


    \begin{corollary}\label{leading order upper profile}
        Suppose $\|f(x,u)\|\leq C\|(x,u)\|^q$ near the origin.  If a feedback $u$ satisfying \eqref{question rate} produces a closed-loop field satisfying $\Omega_{F_u}(r)\geq cr^p$ near $0$, then $d(r)\gtrsim r^{p/q}$ whenever $p<q$.  In particular, if the desired closed-loop vector field has a linear openness profile and $q>1$, then every admissible gain bound satisfies $d(r)\gtrsim r^{1/q}$.
    \end{corollary}


    \begin{proof}
        The hypothesis $\|f(x,u)\|\leq C\|(x,u)\|^q$ implies $\Omega_f(r)\leq Cr^q$ by Lemma~\ref{flat upper profile}.  The result then follows from Corollary~\ref{power law}.
    \end{proof}


    Consequently, Brockett-type data gives explicit exponent restrictions for smooth exponential stabilization.


    \begin{theorem}\label{smooth exponential gain obstruction}
        Suppose $\|f(x,u)\|\leq C\|(x,u)\|^q$ near the origin for some $q>1$.  If a feedback $u$ satisfying \eqref{question rate} produces a closed-loop field $F_u$ that is $C^1$ near $0$ with $DF_u(0)$ invertible, then $d(r)\gtrsim r^{1/q}$.  In particular, this conclusion applies whenever $F_u$ is $C^1$ and has a Hurwitz linearization at the origin.
    \end{theorem}


    \begin{proof}
        By Corollary~\ref{hurwitz gives linear profile}, the closed-loop field satisfies $\Omega_{F_u}(r)\geq cr$ for some $c>0$ and all sufficiently small $r>0$.  Applying Corollary~\ref{leading order upper profile} with $p=1$ gives $d(r)\gtrsim r^{1/q}$.  If $DF_u(0)$ is Hurwitz, then $DF_u(0)$ is invertible, so the same conclusion follows.
    \end{proof}


    \begin{remark}
        The condition $DF_u(0)$ invertible should not be interpreted as a hidden assumption on the feedback; it is a condition on the achieved closed-loop vector field.  The feedback itself may be nonsmooth or even discontinuous, provided that the composition $F_u(x)=f(x,u(x))$ is a classical closed-loop vector field satisfying the stated regularity and stability hypotheses.
    \end{remark}



\section{Norm Estimates and Inverse-Profile Constraints}


    The previous section used upper bounds on the openness profile of the system vector field to derive lower bounds on admissible feedback gains.  There is a complementary interpretation of the same phenomenon.  If a subset of the domain of $f$ is sufficiently large for its image to contain a ball of velocity values around $f(x^*)$, then the set itself must extend sufficiently far from $x^*$.  In the feedback setting, this means that achieving a prescribed closed-loop openness profile requires the graph $x\mapsto (x,u(x))$ to extend a corresponding distance in the joint state-control space.  Since the state component is already bounded by $r$, any remaining distance must be contributed by the control component.

    For this section, when the base point is not the origin, we write
    $$
        \Omega_{f,x^*}(r):=\Gamma_{f(x^*)}\big(f(\mathbb{B}_r(x^*))\big).
    $$
    When $x^*=0$ and $f(0)=0$, this coincides with the openness profile used above.  Also, write $\|\restr{f}{K}\|:=\sup_{x\in K}\|f(x)\|$.


    \begin{lemma}\label{norm est}
        Let \(X,Y\) be normed vector spaces, let \(x^*\in X\), and let
        \(f:X\to Y\). Suppose \(K\subset X\) is a neighborhood of \(x^*\). Then,
        \[
            \|\restr{f}{K}\|
            \geq
            \|f(x^*)\|+\Omega_{f,x^*}(\Gamma_{x^*}(K)).
        \]
    \end{lemma}


    \begin{proof}
        Put \(R:=\Gamma_{x^*}(K)\). If \(\Omega_{f,x^*}(R)=0\), then the conclusion follows from $x^*\in K$. So, assume $\Omega_{f,x^*}(R)>0$, and let $0<\eta<\Omega_{f,x^*}(R)$.  By the definition of \(\Omega_{f,x^*}(R)\),
        \[
            \mathbb B_{\eta}\big(f(x^*)\big)
            \subseteq
            f(\mathbb B_R(x^*)).
        \]
        Since \(\mathbb B_R(x^*)\subseteq K\), it follows that $\mathbb B_{\eta}\big(f(x^*)\big)\subseteq f(K)$.  Now, choose a unit vector \(v\) in the direction of \(f(x^*)\) when
        \(f(x^*)\neq 0\), and choose any unit vector \(v\in Y\) otherwise. For every
        \(0<\theta<\eta\), we have
        \[
            f(x^*)+\theta v\in \mathbb B_{\eta}\big(f(x^*)\big)\subseteq f(K).
        \]
        Hence, $\|\restr{f}{K}\| \geq \|f(x^*)+\theta v\| = \|f(x^*)\|+\theta$, and letting \(\theta\uparrow \eta\) gives $\|\restr{f}{K}\| \geq \|f(x^*)\|+\eta$.  Finally, letting \(\eta\uparrow \Omega_{f,x^*}(R)\) gives the result.
    \end{proof}


    In particular, the openness profile gives a direct lower estimate on the size of a map over any neighborhood of the point about which the profile is taken  For instance, if $f:\mathbb{R}\rightarrow \mathbb{R}$ is given by $f(x)=2x+3$, $x^*=1$, and $K=[0,2]$, then $\Gamma_{x^*}(K)=1$, $f(x^*)=5$, and $\Omega_{f,x^*}(1)=2$.  Lemma~\ref{norm est} gives $\|\restr{f}{K}\|\geq 5+2=7$, which is sharp.

    Similarly, if $f:\mathbb{R}^2\rightarrow \mathbb{R}^2$ is given by $f(x)=Ax+b$ with
    $$
        A=\begin{bmatrix}2&0\\0&1\end{bmatrix},\qquad b=\begin{bmatrix}3\\4\end{bmatrix},
    $$
    and $K=\overline{\mathbb{B}}_1(0)$, then $\Omega_{f,0}(1)=\sigma_{\min}(A)=1$ and $\|f(0)\|=5$.  Therefore Lemma~\ref{norm est} yields $\|\restr{f}{K}\|\geq 6$, though not sharply.


    The same idea may be viewed in reverse.  If the image of a set contains a ball of prescribed radius, then the set itself must extend at least as far as the corresponding generalized inverse of the openness profile.  It is this perspective that yields the connection to the gain-limited stabilization problem of Question~\ref{question}.


    \begin{corollary}\label{reverse norm cor}
        Let \(X,Y\) be normed spaces, let \(x^*\in X\), let \(f:X\to Y\) be a function, and let \(s\geq 0\). Define
        \[
            \Omega_{f,x^*}^{\leftarrow}(s)
            :=
            \inf\{r>0:\ \Omega_{f,x^*}(r)\geq s\},
        \]
        where \(\inf\emptyset:=\infty\). If \(S\subset X\) satisfies $\mathbb B_s(f(x^*))\subset f(S)$, then
        \[
            \sup_{x\in S}\|x-x^*\|
            \geq
            \Omega_{f,x^*}^{\leftarrow}(s).
        \]
    \end{corollary}


    \begin{proof}
        Let $R:=\sup_{x\in S}\|x-x^*\|$.  If $R=\infty$, there is nothing to prove. Otherwise, for every $\epsilon>0$, $S\subseteq \mathbb B_{R+\epsilon}(x^*)$.  Since \(\mathbb B_s(f(x^*))\subset f(S)\), it follows that
        \[
            \mathbb B_s(f(x^*))
            \subset f(S)
            \subseteq f(\mathbb B_{R+\epsilon}(x^*)).
        \]
        Therefore \(s\leq \Omega_{f,x^*}(R+\epsilon)\). By the definition of $\Omega_{f,x^*}^{\leftarrow}$, this yields $\Omega_{f,x^*}^{\leftarrow}(s)\leq R+\epsilon$. Letting $\epsilon\downarrow 0$ gives $\Omega_{f,x^*}^{\leftarrow}(s)\leq R$, which is the desired inequality.
    \end{proof}


    Applying Corollary~\ref{reverse norm cor} to the graph of a feedback recovers the gain obstruction in an inverse form.  This formulation is often the most transparent: in order to produce a closed-loop velocity ball of radius $\gamma(r)$, the graph of the feedback must reach radius at least $\Omega_f^{\leftarrow}(\gamma(r))$ in the joint state-control space.


    \begin{theorem}[Inverse-profile gain obstruction]\label{inverse profile obstruction}
        Let $f:\mathbb{R}^n\times\mathbb{R}^m\rightarrow \mathbb{R}^n$ be continuous with $f(0,0)=0$, let \(d:[0,\epsilon)\to[0,\infty)\) be nondecreasing, and assume $u$ is a feedback satisfying $\|u(x)\|\leq d(\|x\|)$ near $0$.  Set $F_u(x)=f(x,u(x))$.  If the closed-loop field is open with rate $\gamma$ for all sufficiently small $r>0$, then
        $$
            \Omega_f^{\leftarrow}(\gamma(r))\leq \sqrt{r^2+d(r)^2}
        $$
        for all sufficiently small $r>0$.  In particular,
        $$
            d(r)\geq \sqrt{\left(\Omega_f^{\leftarrow}(\gamma(r))\right)^2-r^2}
        $$
        whenever the expression on the right is real.
    \end{theorem}


    \begin{proof}
        Since $\Omega_{F_u}(r)\geq \gamma(r)$, we have
        $$
            \mathbb{B}_{\gamma(r)}(0)\subseteq F_u(\mathbb{B}_r(0))=f(G_u(\mathbb{B}_r(0))),
        $$
        where $G_u(x)=(x,u(x))$.  By Corollary~\ref{reverse norm cor}, applied to $f$ at $(0,0)$ and to the set $S=G_u(\mathbb{B}_r(0))$, it follows that
        $$
            \sup_{\|x\|\leq r}\|G_u(x)\|\geq \Omega_f^{\leftarrow}(\gamma(r)).
        $$
        On the other hand, the gain bound (see the proof of Theorem~\ref{graph limited brocket inequality}) gives
        $$
            \sup_{\|x\|\leq r}\|G_u(x)\|\leq \sqrt{r^2+d(r)^2}.
        $$
        Combining these two inequalities gives
        $$
            \Omega_f^{\leftarrow}(\gamma(r))\leq \sqrt{r^2+d(r)^2}.
        $$
        The stated lower bound on $d$ follows by rearranging.
    \end{proof}

    This inverse formulation is a geometric reformulation of the same obstruction expressed in  Corollary~\ref{minimum gain from closed-loop openness}, and is sometimes easier to interpret.  For example, if $\Omega_f(r)\leq Cr^q$, then $\Omega_f^{\leftarrow}(s)\geq (s/C)^{1/q}$.  Substituting $s=\gamma(r)=cr^p$ into Theorem~\ref{inverse profile obstruction} gives
    $$
        d(r)\geq \sqrt{\left(\frac{c}{C}\right)^{2/q}r^{2p/q}-r^2},
    $$
    which is precisely the power-law obstruction in Corollary~\ref{power law}.  Thus the exponent restriction $d(r)\gtrsim r^{p/q}$ admits a natural geometric interpretation: the graph of a gain-limited feedback must extend sufficiently far in the control directions to access the part of the open-loop vector field capable of generating the desired closed-loop velocity ball.


    \begin{remark}
        The estimates in this section are not intended as separate sufficient conditions for stabilization.  Rather, they are geometric consequences of openness.  Their role is to translate a desired closed-loop openness profile into a minimum radius that must be attained in the joint state-control space.  Theorem~\ref{graph limited brocket inequality} expresses this restriction in direct form, while Theorem~\ref{inverse profile obstruction} provides the same restriction in inverse-profile formulation.
    \end{remark}


\section{Conclusion}


    Brockett's condition is not merely a binary yes/no obstruction to stabilization. Its quantitative openness profile controls how much closed-loop openness can be achieved by a feedback whose graph is constrained by a prescribed gain bound.  Theorem~\ref{graph limited brocket inequality} and Theorem~\ref{inverse profile obstruction} show that the openness profile of the closed-loop vector field $F_u(x)=f(x,u(x))$ must fit inside the openness profile of the original system vector field evaluated at the radius accessible to the graph $x\mapsto (x,u(x))$.  Consequently, any prescribed lower openness rate for the closed-loop dynamics imposes a corresponding lower bound on the growth of the feedback.

    In the power-law setting, this relationship admits a particularly transparent interpretation.  If $\Omega_f(r)\lesssim r^q$ while the desired closed-loop field satisfies $\Omega_{F_u}(r)\gtrsim r^p$, then every admissible gain bound must satisfy $d(r)\gtrsim r^{p/q}$ when $p<q$.  In particular, smooth exponential stabilization corresponds to a linear closed-loop profile. Thus, systems whose Brockett profile is of order $r^q$ require feedbacks whose magnitude is at least of order $r^{1/q}$.  The signed-power examples show that these exponents are sharp.

    Thus the openness datum appearing in Brockett's theorem contains more information than the usual binary obstruction alone.  It also encodes quantitative limitations on the gain profiles of stabilizing feedbacks.
    

\backmatter

\section*{Declarations}

\begin{itemize}
\item Funding: The authors did not receive support from any organization for the submitted work.
\item Competing interests:  All authors certify that they have no affiliations with or involvement in any organization or entity with any financial interest or non-financial interest in the subject matter or materials discussed in this manuscript.
\end{itemize}

\bibliography{ref}

\end{document}